\documentclass[11pt]{amsart}
\usepackage[a4paper]{geometry}
\usepackage{graphicx}
\usepackage{amsfonts}
\usepackage{amsmath}
\usepackage{amssymb}
\usepackage{amsthm}
\usepackage{newlfont}
\usepackage{yhmath}
\usepackage{color}

\newtheorem{lemma}{Lemma}[section]

\newtheorem{theorem}{Theorem}[section]

\newtheorem{rmk}{Remark}[section]

\allowdisplaybreaks

\usepackage{tikz}
\usetikzlibrary{positioning,shapes,arrows}
\usepackage{ytableau}  
\allowdisplaybreaks         
\ytableausetup{boxsize=1em}   
\begin{document}

\title[Precise Deviations for the Ewens-Pitman Model]{Precise Deviations for the Ewens-Pitman Model}

\author{Zhiqi Peng}
\address{School of Mathematics and Physics, Xi’an Jiaotong-Liverpool University, Suzhou 215123, China}
\email{zhiqi.peng16@gmail.com}

\author{Youzhou Zhou}
\address{School of Mathematics and Physics, Xi’an Jiaotong-Liverpool University, Suzhou 215123, China}
\email{youzhou.zhou@xjtlu.edu.cn}
\thanks{Supported by NSFC Grants 12471146 and RDF-22-01-013.}



\begin{abstract}
In this paper, we derive an integral representation for the distribution of the number of types $K_n$ in the Ewens-Pitman model. Based on this representation, we also establish precise large deviations and precise moderate deviations for $K_n$. After careful examination, we find that the rate function exhibits a second-order phase transition and the critical point is $\alpha=\frac{1}{2}$. 
 \end{abstract}

\date{\today}

\maketitle  

\section{Introduction}

Ewens-Pitman sampling model is a two-parameter generalization of the famous Ewens sampling formula \cite{E} in population genetics. Kingman discovered the consistency structure \cite{K1} among the Ewens sampling distributions and introduced the partition structures in \cite{K2}. Partition structures are just a family of random partition distributions that satisfy certain consistent conditions; their characterization is given by an integral representation
$$
\int_{\overline{\nabla}_{\infty}}P_{\eta}(x)\mu(dx), 
$$
where $\overline{\nabla}_{\infty}=\{x\in[0,1]^{\infty}\mid \sum_{i=1}^{\infty}x_i\leq 1, x_1\geq x_2\geq\cdots\geq0\}$ is the Kingman simplex and $P_\eta(x)$ is just the sampling probability of a partition $\eta$ from an urn with color frequency $x\in \overline{\nabla}_{\infty}$.

The Ewens-Pitman sampling model is a special partition structure; its representation measure $\mu(dx)$ is the two-parameter Poisson-Dirichlet distribution $PD(\alpha,\theta), \alpha\in(0,1), \theta+\alpha>0$ (see \cite{F2}). Thus, we can replicate the distribution of the Ewens-Pitman sampling model by a two-stage sampling scheme. First, we randomly select an urn from  $PD(\alpha,\theta)$, meaning its color frequency follows the two-parameter Poisson-Dirichlet distribution $PD(\alpha,\theta)$. Then we sample with replacement to obtain a sample whose frequency counts follow the Ewens-Pitman sampling distribution.

Moreover, like P\'olya urn model, the Ewens-Pitman sampling model can also be realized through the following sequential sampling scheme:
\begin{equation}
\mathbb{P}(X_1\in \cdot)=\nu(\cdot),~~\mathbb{P}(X_{n+1}\cdot\mid X_1,\cdots,X_n)=\frac{\theta+k\alpha}{\theta+n}\nu(\cdot)+\frac{1}{\theta+n}\sum_{i=1}^k(n_l-\alpha)\delta_{X_i^*}(\cdot) \label{sampling_model}
\end{equation}
where $X_1^*,\cdots,X_k^*$ are the distinct values in $(X_1,\cdots,X_n)$ and $(n_1,\cdots,n_k)$ is the frequency counts. Therefore, the Ewens-Pitman model has many applications in Bayesian statistics. 

Let $X_1,\cdots,X_n$ be a random sample generated from the sampling scheme (\ref{sampling_model}). Besides the frequency counts $(N_1,\cdots,N_{K_n})$, there are other important statistics, such as
$$
M_{ln}=\sum_{i=1}^{K_n}\mathbb{I}_{\{N_{i,n}=l\}}, ~1\leq l\leq n, ~\text{and}~K_n=\sum_{l=1}^nM_{ln}.
$$
The joint distribution of $(M_{1n},\cdots,M_{nn},K_n)$ is 
$$
\mathbb{P}(M_{1n}=m_1,\cdots,M_{nn}=m_n, K_n=k)=n!\frac{\prod_{i=1}^{k}(\theta+(i-1)\alpha)}{\theta_{n\uparrow 1}} \prod_{j=1}^n \left( \frac{(1-\alpha)_{j-1 \uparrow 1}}{j!} \right)^{\!\! m_j}  \!\! \!\! \frac{1}{m_j!} ,
$$
where $(a)_{z\uparrow 1}=\frac{\Gamma(a+z)}{\Gamma(a)}$.  Here the statistics $K_n$ indicates the overall diversity of the sample $X_1,\cdots,X_n$. As $n\to\infty$, 
\begin{equation}
\frac{K_n}{n^{\alpha}}\to S_{\alpha,\theta} ~a.s.\label{fluctuation}
\end{equation}
where $S_{\alpha,\theta}$, known as Pitman's $\alpha$-diversity, follows a tilted Mittag-Leffler distribution (see \cite{P2}). The Berry-Esseen bound of $K_n$, i.e. the upper bound of the Kolmogorov distance between $\frac{K_n}{n^{\alpha}}$ and $S_{\alpha,\theta}$, was obtained by  Dolera and Favaro in \cite{DF}. It is also easy to see that $\lim_{n\to\infty}\frac{K_n}{n}=0.$ If we view (\ref{fluctuation}) as a fluctuation result for $K_n$, then $\lim_{n\to\infty}\frac{K_n}{n}=0$ can be regarded as a law of large numbers for $K_n$. The large deviation principle for $K_n$ was established in \cite{FH}, and the moderate deviation principle for $K_n$ was obtained in \cite{FFG}. Recently,  concentration inequality for $K_n$ have also been studied (see \cite{BF}). Similar results on large deviations and moderate deviations for other statistics, such as $M_{ln}$, have also been discussed (see \cite{FF1,FF2,F1}). Moreover, $\lim_{n\to\infty}\frac{M_{ln}}{K_n}=\frac{\alpha(1-\alpha)_{(l-1)\uparrow1}}{l!},l\geq1,$ was first considered by Yamato and Sibuya  in \cite{YS}. Now $\{P_{\alpha}(l)=\frac{\alpha(1-\alpha)_{(l-1)\uparrow1}}{l!},l\geq1\}$ is a discrete distribution called Sibuya distribution. 

To the best of the author's knowledge, results on precise deviations for $K_n$ are still lacking. Given the wide applications of the Ewens-Pitman sampling model in Bayesian statistics, precise deviations for $K_n$ play an important role in parameter estimation and hypothesis testing. In this article, we will derive the precise deviation results for $K_n$ through a complex integral representation of the distribution $\mathbb{P}(K_n=k)$. These results will provide new insights into the understanding of the Ewens-Pitman sampling model. 
\subsection{Notations}

In this article, several notations will be repeatedly used. For convenience, we summarize them as follows:
\begin{itemize}
\item $(a)_{z\uparrow 1}=\frac{\Gamma(a+z)}{\Gamma(z)}$.
\item We write $f_n=O(g_n)$ if and only if $\sup_{n\geq1}\left|\frac{f_n}{g_n}\right|<\infty$.  If $\lim_{n\to\infty}\frac{f_n}{g_n}=1$, then we write $f_n\sim g_n$. 
\item We write $a_n=b_n(1+O(c_n))$ if and only if 
$$
\sup_{n\geq1}|(\frac{a_n}{b_n}-1)/c_n|\leq M<\infty
$$
where $M$ is independent of $n$. 
\item $\mathrm{i}=\sqrt{-1}$ denotes the imaginary unit. For a complex number $z=x+\mathrm{i}y$, we denote its real part $x$ as $\mathrm{Re}(z)$ and its imaginary part $y$ as $\mathrm{Im}(z)$.
\item For a real number $x$,  its ceiling $\lceil x\rceil$ is defined as $\lceil x\rceil=\min\{n\mid n\geq x\}$, and  its floor $\lfloor x\rfloor$ is defined as $\lfloor x\rfloor=\max\{n\mid n\leq x\}$. The fractional part of $x$ is defined as $\{x\}=\lceil x\rceil -x$ and we have $|\{x\}|\leq 1$.
\end{itemize}

\subsection{Main Results}

Large deviations and moderate deviations for $K_n$ concern the estimation of tail probabilities. 
Since $\lim_{n\to\infty}\frac{K_n}{n}=0$ is the law of large number, large deviations deal with the estimation of 
$$
\log \mathbb{P}(K_n\geq nx)\sim -nI(x)
$$
where $n$ is the speed and $I(x)$ is the rate function. By (\ref{fluctuation}), the fluctuation scale is $n^{\alpha}$; hence   moderate deivations should exhibit the following tail behaviour
$$
\log \mathbb{P}(K_n\geq n^{\alpha}b_n^{1-\alpha}x)\sim -b_nJ(x)
$$
where $b_n$ is the speed and $J(x)$ is called rate function of moderate deivations for $K_n$. Here $b_n$ is the speed between fluctuation and the law of large numbers, so $\lim_{n\to\infty}b_n=\infty$ and $\lim_{n\to\infty}\frac{b_n}{n}=0$. 

Precise deviations, on the other hand, directly provide the precise tail behaviour of $\mathbb{P}(K_n\geq nx)$ and $\mathbb{P}(K_n\geq n^{\alpha} b_n^{1-\alpha}x)$ without first taking the logarithm in the first place. Thus, precise deviations retain many detailed information of tail probabilities, which are often crucial for statistical inference.   

Next we will summarize two main theorems of this article.

 \begin{theorem}[Precise LDP of $K_n$]
 As $n\to\infty$, we have 
 $$
\mathbb{P}(K_n\geq xn)=\frac{\Gamma(\theta)}{\Gamma(\theta/\alpha)}\frac{n^{(\frac{1}{\alpha}-1)\theta-\frac{1}{2}}}{z(x)\sqrt{2\pi |h''(z(x))|}}x^{\frac{\theta}{\alpha}-1}\frac{e^{-\{nx\}I'(x)}}{1-e^{-I'(x)}}e^{-nI(x)}\left[1+O(\frac{1}{x}\frac{1}{n})\right]
$$
where $h(z)=\ln z-x\ln(1-(1-z)^{\alpha})$ and $z(x)$ is the solution of $h'(z)=0,z\in(0,1)$.  Here $I(x)=h(z(x))$ is the rate function of LDP for $K_n$. 
 \end{theorem}

From this result, we can see that the optimal polynomial scale should be $n^{(\frac{1}{\alpha}-1)\theta-\frac{1}{2}}$.  

 \begin{theorem}[Precise MDP for $K_n$]
 For $x=y\frac{n^{\alpha}b_n^{1-\alpha}}{n}$, where $\lim_{n\to\infty}b_n=\infty$ and $\lim_{n\to\infty}\frac{b_n}{n}=0$, we have 
 \begin{align*}
&\mathbb{P}(K_n\geq yn^{\alpha}b_n^{1-\alpha})\\
=&\frac{\Gamma(\theta)}{\Gamma(\theta/\alpha)}\frac{b_n^{(1-\alpha)\frac{\theta}{\alpha}-\frac{1}{2}}}{\sqrt{2\pi(1-\alpha)\alpha^{\frac{2\alpha-1}{1-\alpha}}}}y^{\frac{\theta}{\alpha}-\frac{1}{2(1-\alpha)}}e^{-nI(y(\frac{b_n}{n})^{1-\alpha})}\left[1+O\left[(y^{\frac{\alpha}{1-\alpha}}\vee \frac{1}{y}\frac{1}{b_n})(\frac{b_n}{n})^{\alpha}\right]\right]
\end{align*}
 \end{theorem}

Based on this result, we also see that the optimal polynomial scale is $b_n^{(1-\alpha)\frac{\theta}{\alpha}-\frac{1}{2}}$, which does not depend on $n$ at all. Therefore, unlike the argument in \cite{FFG}, we do not need to impose any restrictions on $b_n$, and the moderate deviation principle(MDP) for $K_n$ holds for all moderate deviation scales $b_n$ as long as $\lim_{n\to\infty}b_n=\infty$ and $\lim_{n\to\infty}\frac{b_n}{n}=0$.  Furthermore, the term $\alpha^{\frac{2\alpha-1}{1-\alpha}}$ in the coefficient $\frac{\Gamma(\theta)}{\Gamma(\theta/\alpha)}\frac{1}{\sqrt{2\pi(1-\alpha)\alpha^{\frac{2\alpha-1}{1-\alpha}}}}$ exihbits a phase transition, because the sign of $\frac{2\alpha-1}{1-\alpha}$ is completely different for $\alpha\in (0,\frac{1}{2})$ and $\alpha\in (\frac{1}{2},1)$. In fact, by Lemma \ref{MDP_rate}, the second derivative of the rate function $I(x)$ also displays similar phase transition behaviour as $x\to 0$:
$$
I''(x)\sim \frac{\alpha^{\frac{1}{1-\alpha}}}{1-\alpha}x ^{\frac{2\alpha-1}{1-\alpha}}\to
\begin{cases}
+\infty & \alpha\in(0,\frac{1}{2})\\
\frac{\alpha^{\frac{1}{1-\alpha}}}{1-\alpha}=\frac{1}{2}& \alpha=\frac{1}{2}\\
0 & \alpha\in(\frac{1}{2},1)
\end{cases}
$$
This is clearly a second-order phase transition. 

In this article, we discuss the integral representation of $\mathbb{P}(K_n=k)$ in Section 2. Using this representation, we derive the local precise deviations in Section \ref{local} and obtain the global precise deviations in Section \ref{global}. Finally, we conclude the article with a remark in Section 5.

\section{Integral Representation of $\mathbb{P}(K_n=k)$}\label{integral}

In this section, we will obtain an integral representation for the probability mass function of $K_n$. To this end, let us first recall the joint distribution of  $(M_{1n},\cdots,M_{nn},K_n)$,  
$$
\mathbb{P}(M_{1n}=m_1,\cdots,M_{nn}=m_n, K_n=k)=n!\frac{\prod_{i=1}^{k}(\theta+(i-1)\alpha)}{\theta_{n\uparrow 1}} \prod_{j=1}^n \left( \frac{(1-\alpha)_{j-1 \uparrow 1}}{j!} \right)^{\!\! m_j}  \!\! \!\! \frac{1}{m_j!} .
$$
Thus, the distribution of $K_n$ is the marginal distribution of $(M_{1n}=m_1,\cdots,M_{nn}=m_n, K_n=k)$. Then
		\begin{align*}
			\mathbb{P}(K_n=k)=& \sum_{\substack{\sum_{j=1}^{n} m_j = k \\ \sum_{j=1}^{n} j m_j = n}} n! \frac{\prod_{i=1}^{k}(\theta + (i-1)\alpha)}{\theta_{n \uparrow 1}} \prod_{j=1}^n \left( \frac{(1-\alpha)_{j-1 \uparrow 1}}{j!} \right)^{m_j} \frac{1}{m_j!} \\
			= &  \frac{n!}{k!} \frac{\left( \frac{\theta}{\alpha} \right)_{k \uparrow 1}}{\theta_{n \uparrow 1}} \sum_{\substack{\sum_{j=1}^{n} m_j = k \\ \sum_{j=1}^{n} j m_j = n}} \frac{k!}{m_1! \cdots m_n!} \prod_{j=1}^n \left( \frac{\alpha (1-\alpha)_{j-1 \uparrow 1}}{j!} \right)^{m_j},
		\end{align*}
	where $\{P_{\alpha}(j)=\frac{\alpha(1-\alpha)_{j-1 \uparrow 1}}{j!} , j\geq1\}$ is also known as Sibuya distribution \cite{YS}.  Let $\{X_i,1\leq i\leq k\}$ be a sequence of independent random variables and they identically follow the Sibuya distribution. One can easily show that 
	$$
	\mathbb{P}\left(\sum_{i=1}^kX_i=n\right)=\sum_{\substack{\sum_{j=1}^{n} m_j = k \\ \sum_{j=1}^{n} j m_j = n}} \frac{k!}{m_1! \cdots m_n!} \prod_{j=1}^n \left( \frac{\alpha (1-\alpha)_{j-1 \uparrow 1}}{j!} \right)^{m_j}.
	$$	

	Thus, we have established the following lemma. 
\begin{lemma}
		Let $ \left\{   X_l,l\geq1  \right\}   $ be a sequence of  $ \,i.i.d.$ random variables valued in $\{1,2,\cdots\}$ and commonly follows the Sibuya distribution,i.e.
			$$
			\mathbb{P}(X_1=j)=P_{\alpha}(j)=\frac{\alpha (1-\alpha)_{j-1 \uparrow 1}}{j!},
			$$
Then
			$$
			\mathbb{P}(K_n = k) = \frac{n!}{k!} \frac{\left( \frac{\theta}{\alpha} \right)_{k \uparrow 1}}{\theta_{n 	\uparrow 1}} \mathbb{P}\left( \sum_{l=1}^{k} X_l = n \right).
			$$
\end{lemma}
	
Making use of the generating function of $\sum_{l=1}^{k}X_l$, we can also derive an integral representation for $\mathbb{P}\left(\sum_{l=1}^kX_l=n\right)$. In the end, we will end up with an integral representation of the probability mass function of $K_n$.
	\begin{lemma} \label{integral_representation}
		The probability mass function of  $K_n$ has the following integral representation:
		\begin{equation}
			\mathbb{P}(K_n = k) = \frac{n!}{k!} \frac{\left( \frac{\theta}{\alpha} \right)_{k\uparrow 1}}{(\theta)_{n\uparrow1 }} \frac{1}{2\pi \mathrm{i}} \int_{C} \frac{(1 - (1 - z)^{\alpha})^k}{z^{n+1}} dz,  \label{integral}
		\end{equation}
		where $C$ is a circle $|z|=r < 1$.
	\end{lemma}
	\begin{proof}
	The generating formula of $\sum_{l=1}^{k}X_l$
	$$
	G_{\sum_{l=1}^{k}X_l}(z)=\mathbb{E}z^{\sum_{l=1}^{k}X_l}=[\mathbb{E}z^{X_1}]^k,
	$$
	where
	$$
	\mathbb{E}z^{X_1}=\sum_{j=1}^{\infty}z^jP_{\alpha}(j)=\sum_{j=1}^{\infty}z^j\frac{\alpha (1-\alpha)_{j-1 \uparrow 1}}{j!}=1-(1-z)^{\alpha}.
	$$
Therefore, the generating formula of $\sum_{l=1}^{k}X_l$ is 
$$
G_{\sum_{l=1}^{k}X_l}(z)=\mathbb{E}z^{\sum_{l=1}^{k}X_l}=[1-(1-z)^{\alpha}]^k.
$$
By Cauchy formula, we know 
$$
\mathbb{P}(\sum_{l=1}^{k}X_l=n)=\frac{1}{2\pi i}\int_{C}\frac{[1-(1-z)^{\alpha}]^k}{z^{n+1}}dz,
$$
where $C$ is a circle $|z|=r<1$.
\end{proof}

\section{Local Precise Deviations}\label{local}
Fluctuations, large deviations, and moderate deviations for $K_n$ have already been established; see \cite{DF,FF1,FF2,FFG,FH}. We already know that the fluctuation scale is $n^{\alpha}$, and the large deviation scale is $n$. Therefore, the moderate deviation scale must be $n^{\alpha}b_n^{1-\alpha}$, where $\lim_{n\to\infty}b_n=+\infty$ and $\lim_{n\to\infty}\frac{b_n}{n}=0$.

In this section, we derive the local precise deviations. Using the integral representation of $K_n$ in Lemma \ref{integral_representation}, the coefficients $\frac{n!}{k!} \frac{\left( \frac{\theta}{\alpha} \right)_{k\uparrow 1}}{(\theta)_{n\uparrow1 }}$ in the formula are easily handled.

\begin{lemma}\label{coefficient_asymptotic}
Let $x_k=\frac{k}{n}$, then, as $n\to\infty$, we have
$$
\frac{n!}{k!} \frac{\left( \frac{\theta}{\alpha} \right)_{k\uparrow 1}}{(\theta)_{n\uparrow1 }}=\frac{\Gamma(\theta)}{\Gamma(\frac{\theta}{\alpha})}n^{(\frac{1}{\alpha}-1)\theta}x_k^{\frac{\theta}{\alpha}-1}(1+O(\frac{1}{k}))
$$
\end{lemma}

\begin{proof}
By Stirling's formula
$$
\Gamma(z)=\sqrt{\frac{2\pi}{z}}(\frac{z}{e})^z(1+O(\frac{1}{z})),
$$
we have
\begin{align*}
&\frac{n!}{k!} \frac{\left( \frac{\theta}{\alpha} \right)_{k\uparrow 1}}{(\theta)_{n\uparrow1 }}
=\frac{\Gamma(n+1)}{\Gamma(k+1)}\frac{\Gamma(\frac{\theta}{\alpha}+k)/\Gamma(\frac{\theta}{\alpha})}{\Gamma(\theta+n)/\Gamma(\theta)}\\
=&\frac{\Gamma(\theta)}{\Gamma(\frac{\theta}{\alpha})}\frac{\sqrt{\frac{2\pi}{n+1}}\left(\frac{n+1}{e}\right)^{n+1}(1+O(\frac{1}{n}))\sqrt{\frac{2\pi}{k+\theta/\alpha}}\left(\frac{\theta/\alpha+k}{e}\right)^{k+\theta/\alpha}(1+O(\frac{1}{k}))}{\sqrt{\frac{2\pi}{k+1}}\left(\frac{k+1}{e}\right)^{k+1}(1+O(\frac{1}{k}))\sqrt{\frac{2\pi}{n+\theta}}\left(\frac{\theta+n}{e}\right)^{n+\theta}(1+O(\frac{1}{n}))}\\
=&\frac{\Gamma(\theta)}{\Gamma(\frac{\theta}{\alpha})}n^{(\frac{1}{\alpha}-1)\theta}x_k^{\frac{\theta}{\alpha}-1}(1+O(\frac{1}{k}))
\end{align*}
\end{proof}

Next, we deal with the second part $\frac{1}{2\pi \mathrm{i}} \int_{C} \frac{(1 - (1 - z)^{\alpha})^k}{z^{n+1}} \, dz$ in the integral representation of $\mathbb{P}(K_n=k)$. Thus, we first determine the saddle point and then perform contour deformation. Since
$$
\frac{(1 - (1 - z)^{\alpha})^k}{z^{n+1}}=\frac{1}{z}\exp\{-nh(z)\},
$$
where $h(z)=\ln(z)-x_k\ln\!\big(1-(1-z)^{\alpha}\big)$, the saddle points are the roots of the equation
\begin{equation}
0=h'(z)=\frac{1}{z}-x_k\frac{\alpha (1-z)^{\alpha-1}}{1-(1-z)^{\alpha}}.\label{critical_equation}
\end{equation}
The equation (\ref{critical_equation}) may have more than one root; the asymptotic behavior of 	
$$
\frac{1}{2\pi \mathrm{i}} \int_{C} \frac{(1 - (1 - z)^{\alpha})^k}{z^{n+1}} \,dz
$$ depends on the steepest descent contour through the saddle point $z^*$ at which $\mathrm{Re}(h(z))$ attains its minimum among all saddle points. Since $\mathrm{Re}(h(z))$ is symmetric with respect to the real axis and harmonic on the upper half-plane, its minimum can only occur on the boundary, i.e., at infinity or on the real axis. However, $\lim_{z\to\infty}\mathrm{Re}(h(z))=+\infty$, so we need only consider saddle points on the real axis. Note that $(1,+\infty)$ is the branch cut of $(1-z)^{\alpha}$, and there are no saddle points on $(1,+\infty)$. We now apply the transformation $z=\frac{w}{w+1}$, $w>-1$. The critical equation (\ref{critical_equation}) becomes
\begin{equation}
\frac{(1+w)}{w[(1+w)^{\alpha}-1]}\big[(1+w)^{\alpha}-1-\alpha x_kw\big]=0. \label{transformed_critical_equation}
\end{equation}
Since $\frac{(1+w)}{w[(1+w)^{\alpha}-1]}>0$ for $w>-1$, the saddle points must satisfy $(1+w)^{\alpha}-1-\alpha x_kw=0$. To this end, consider $f(w)=(1+w)^{\alpha}-1-\alpha x_kw$. Its derivative is
$$
f'(w)=\alpha(1+w)^{\alpha-1}-\alpha x_k=\alpha\!\left[\frac{1}{(1+w)^{1-\alpha}}-x_k\right]\begin{cases}
>0, & w<\frac{1}{x_k^{1/(1-\alpha)}}-1,\\[4pt]
=0, & w=\frac{1}{x_k^{1/(1-\alpha)}}-1,\\[4pt]
<0, & w>\frac{1}{x_k^{1/(1-\alpha)}}-1.
\end{cases}
$$
Thus $f(w)=0$ has at most two real roots. One is $w=0$, and the other, denoted $w(x_k)$, satisfies $w(x_k)>1/(x_k^{1/(1-\alpha)}-1)$. One checks easily that $w=0$ is not a solution of (\ref{transformed_critical_equation}); hence the only root of (\ref{transformed_critical_equation}) is $w(x_k)$, and the unique real saddle point of $h(z)$ is $z^*=z(x_k)=\frac{w(x_k)}{1+w(x_k)}$. Moreover, by the monotonicity of $f$ we have $f(w)>0$ for $w\in(0,w(x_k))$ and $f(w)<0$ for $w\in(w(x_k),\infty)$. Consequently,
$$
h'(z)\begin{cases}
>0, & z\in(0, z^*),\\[2pt]
<0, & z\in(z^*,1),
\end{cases}
$$
so that $h(z^*)=\sup_{z\in(0,1)}h(z)$ and $h''(z^*)<0$. Incidentally, the real axis is the steepest ascent contour, while the line $z=z^*+\mathrm{i}t$, $t\in(-\infty,\infty)$ is the steepest descent line for the integral $\frac{1}{2\pi \mathrm{i}} \int_{C} \frac{(1 - (1 - z)^{\alpha})^k}{z^{n+1}} \, dz$. Hence
\begin{equation}
h(z^*)=\sup_{z\in(0,1)}h(z)=\inf_{t\in\mathbb{R}}\mathrm{Re}\!\big(h(z^*+\mathrm{i}t)\big).\label{variational_representation}
\end{equation}

We are now ready to deform the integration contour to the steepest descent contour through the saddle point $z^*$.

\begin{lemma} \label{Steepest_descent_deformation}  
	\begin{align*}
\frac{1}{2 \pi \mathrm{i}} \int_{C} \frac{\left[1-(1-z)^\alpha\right]^k}{z^{n+1}} d z =  \frac{1}{2\pi \mathrm{i}} \int_{z^*-i\infty}^{z^*+i\infty}  \frac{1}{z} \exp \left\{-n    h(z) \right\} d z ,
	\end{align*}
  where \( h(z)=\ln z-x_k \ln \left[1-(1-z)^\alpha\right]  \) and $C$ is the circle $|z|=r<1$.
	\end{lemma}
\begin{proof}
We consider the following intermediate contour:
\begin{figure}[htbp]
\begin{center}
\includegraphics[width=2.3in, height=2in]{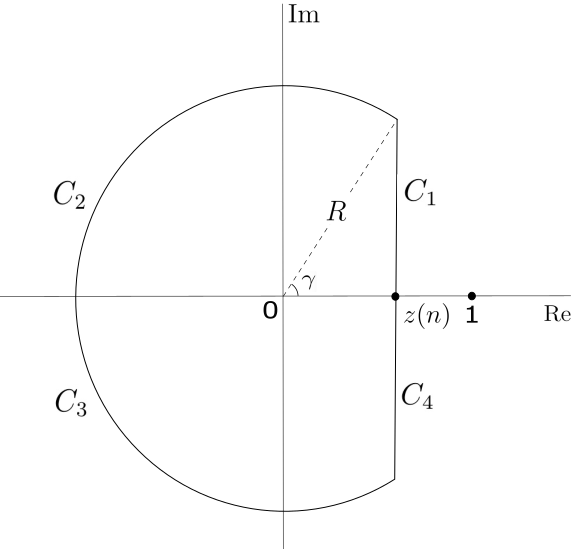}
\caption{Intermediate Steepest Descent Contour}
\label{middle Steepest Descent Contour 1}
\end{center}
\end{figure}
 			\begin{itemize}
 			\item     $C_{1}$:  This  is the upward vertical path starting from  the saddle point $z(n)$ on the real axis, defined as
 			$$z(t) = z(n) + \mathrm{i} t, \quad t \in [0, R \sin \gamma].$$
 			\item  $C_{2}$:   This is an arc of a circle with radius     $R$   from the upper point of  $C_1$ to the left along the circle, parameterized as
 			$$z(t)=R e^{ \mathrm{i} t}, \quad t \in[\gamma, \pi]. $$
 			\item    $C_3: $   This segment mirrors $C_2$ on the lower half of the complex plane, given as
 			$$z(t)=R e^{ \mathrm{i} t}, \quad t \in[-\pi,-\gamma]  . $$
 			\item  $C_{4}$:  This is the  vertical path returning from the   endpoint of $C_{3}$ back to the saddle point $z(n)$ on the real axis, parameterized as
 			$$z(t) = z(n) +  \mathrm{i} t, \quad t \in [-R \sin \gamma, 0].$$
 		\end{itemize}

 	We  analyze the integral over the segments $C_1 \cup C_2$ and  $C_3 \cup C_4$ as follows:
 		\begin{align*}
 			&\frac{1}{2 \pi \mathrm{i}} \int_C \frac{1}{z} e^{-n  h(z)} \, dz = \frac{1}{2 \pi  \mathrm{i} } \int_{C_1 \cup C_2} \frac{1}{z} e^{-n h(z)} \, dz + \frac{1}{2 \pi  \mathrm{i} } \int_{C_3 \cup C_4} \frac{1}{z} e^{-n h(z)} \, dz \\
 			=& \frac{1}{2 \pi  \mathrm{i} } \int_0^{R \sin \gamma} \frac{1}{z(n) +  \mathrm{i} t} e^{-n h(z(n) + \mathrm{i}t)}  \mathrm{i}  \, dt + \frac{1}{2\pi  \mathrm{i} } \int_{\gamma}^{\pi} \frac{1}{Re^{ \mathrm{i} t}} e^{-n h(Re^{ \mathrm{i}t})}  \mathrm{i} R e^{ \mathrm{i}t} \, dt \\
 			& +\!\!  \frac{1}{2 \pi  \mathrm{i} } \!\! \int_{-\pi}^{-\gamma} \!\! \frac{1}{Re^{ \mathrm{i} t}} e^{-n h(R e^{ \mathrm{i} t})}  \mathrm{i}  R e^{ \mathrm{i} t} \, dt + \frac{1}{2 \pi \mathrm{i} } \int_{-R \sin \gamma}^0 \frac{1}{z(n) +  \mathrm{i} t} e^{-n h(z(n) +  \mathrm{i}t)}  \mathrm{i} \, dt .
 		\end{align*}
 	Since $ \left| \frac{1}{R e^{\mathrm{it}}} e^{-n h\left(R e^{\mathrm{i} t}\right)}  R e^{\mathrm{i} t}  \right|  \leq    \frac{\left[    1+(1+R)^{\alpha}\right]^{k}}{R^{n}} \leq \left(    \frac{1+(1+R)^{\alpha}}{R}\right)^{n} \rightarrow 0$ as $R\rightarrow \infty$,  i.e., 
 	$$  \lim_{R\rightarrow \infty} \left(     \frac{1}{2 \pi \mathrm{i}} \int_\gamma^\pi \frac{1}{R e^{\mathrm{i} t}} e^{-n h\left(R e^{\mathrm{i} t}\right)} \mathrm{i} R e^{\mathrm{i} t} d t  +    \frac{1}{2 \pi \mathrm{i}} \int_{-\pi}^{-\gamma} \frac{1}{R e^{\mathrm{i} t}} e^{-n h\left(R e^{\mathrm{i} t}\right)} \mathrm{i} R e^{\mathrm{i} t} d t  \right)  =0  .   $$ 
 	 Therefore, the integral over the  large arc path $C_{2} \cup C_{3}$ vanishes. As a result, we have
	$$\frac{1}{2 \pi \mathrm{i}} \int_C \frac{1}{z} e^{-n h(z)} d z=       \frac{1}{2 \pi \mathrm{i}} \int_{z^*-\mathrm{i}\infty} ^{ z^*+ \mathrm{i}\infty}\frac{1}{z} e^{-n h(z)} d z   ,$$ 
This completes the proof of Lemma \ref{Steepest_descent_deformation}. 
\end{proof}

 \begin{lemma} \label{Asymptotic_Steepest_descent_integral}
As \( n \to \infty \),  we have
\begin{align*}
\frac{1}{2 \pi \mathrm{i} } \int_{z^*- \mathrm{i} \infty}^{z^*+\mathrm{i} \infty} \frac{\left[1-(1-z)^\alpha\right]^k}{z^{n+1}} d z=\frac{1}{z^*\sqrt{2\pi |h''(z^*)|n}}e^{-nh(z^*)}(1+O(\frac{1}{n}))
\end{align*}
\end{lemma}

\begin{proof}
For $z=z^*+ \mathrm{i} t$, where $t \in (-\infty , \infty)$ and   $0 < z^* < 1 $,  we have
		\begin{align*}
			\frac{ \left|1-(1-z)^{\alpha}\right|^{k} }{|z|^{n}} \leq &\frac{ \left|1+|1-z|^{\alpha}\right|^{k} }{|z|^{n}}=\frac{ \left[   1+ \left((1-z^*)^2+t^2\right)^{\frac{\alpha}{2}} \right] ^{k}}{\left((z^*)^2+t^2\right)^{\frac{n}{2}} }    \\  \leqslant & \frac{\left|1+\left(1+t^2\right)^{\frac{\alpha}{2}}\right|^k}{|t|^n}  \leqslant    \left(\frac{2^{1+\frac{\alpha}{2}}}{|t|^{1-\alpha}}\right)^n .   
		\end{align*} 
Then we decompose the integral into three parts as follows:
$$
\frac{1}{2 \pi \mathrm{i} } \int_{z^*- \mathrm{i} \infty}^{z^*+\mathrm{i} \infty} \frac{\left[1-(1-z)^\alpha\right]^k}{z^{n+1}} d z=A+B+C
$$
where 
\begin{align*}
A=& \frac{1}{2 \pi  } \int_{|t|\geq M } \frac{\left[1-(1-z^*-\mathrm{i} t)^\alpha\right]^k}{(z^*+\mathrm{i}t)^{n+1}} d t\\
B=& \frac{1}{2 \pi  } \int_{\delta\leq |t|< M} \frac{\left[1-(1-z^*-\mathrm{i} t)^\alpha\right]^k}{(z^*+\mathrm{i}t)^{n+1}} d t\\
C=& \frac{1}{2 \pi  } \int_{|t|< \delta} \frac{\left[1-(1-z^*-\mathrm{i} t)^\alpha\right]^k}{(z^*+\mathrm{i}t)^{n+1}} d t,
\end{align*}
$M=2^{\frac{1+\alpha/2}{1-\alpha}}e^{\frac{2}{(1-\alpha)}h(z^*)}$, $0<\delta<\frac{z^*}{2}$ and 
$$
\frac{d}{dt}\mathrm{Re}(h(z^*+\mathrm{it}))=\begin{cases}
>0, & 0<t<\delta\\
<0, & -\delta<t<0.
\end{cases}
$$
Note that 
\begin{align*}
|A|\leq & \frac{1}{2M\pi} \int_{|t|\geq M}\left(\frac{2^{1+\frac{\alpha}{2}}}{|t|^{1-\alpha}}\right)^n d t
=\frac{1}{M\pi }\int_{M}^{\infty}\exp\{-n\log \frac{t^{1-\alpha}}{2^{1+\frac{\alpha}{2}}}\}dt
\end{align*}
By substitution $w=n\log \frac{t^{1-\alpha}}{2^{1+\frac{\alpha}{2}}}$, we have
\begin{align*}
|A|\leq &  \frac{1}{(1-\alpha)\pi} \frac{1}{n}\exp\{-n\log\frac{M^{1-\alpha}}{2^{1+\alpha/2}}\}\int_0^{\infty}e^{\frac{v}{(1-\alpha)n}}e^{-v}dv\\
=& \frac{1}{(1-\alpha)\pi} \frac{1}{n}e^{-2nh(z^*)}\sum_{k=0}^{\infty}\frac{1}{((1-\alpha)n)^k} \int_0^{\infty}\frac{v^k}{k!}e^{-v}dv\\
=&  \frac{1}{(1-\alpha)\pi} \frac{1}{n}e^{-2nh(z^*)}e^{\frac{1}{(1-\alpha)n}}=O\left(\frac{1}{n}e^{-2nh(z^*)}\right).
\end{align*}
Moreover, $\mathrm{Re}(h(z))$ is a non-constant harmonic function on upper half plane and lower half plane, then $\inf_{\delta\leq |t|<M}\mathrm{Re}(h(z^*+\mathrm{i}t))>h(z^*)$ because $z^*$ is the minimum point on the boundary. Thus,
\begin{align*}
|B|\leq & \frac{1}{2 \pi  } \int_{\delta\leq |t|<M}\frac{1}{|z^*+\mathrm{i}t|}\exp\left\{-n\mathrm{Re}(h(z^*+\mathrm{i}t))\right\}d t
\leq \frac{1}{\pi}e^{-n\inf_{\delta\leq |t|<M}\mathrm{Re}(h(z^*+\mathrm{i}t))}
\end{align*}
Now for part $C$,  we consider a substitution $w=\sqrt{n|h''(z^*)|}t$, then
\begin{align*}
C=&\frac{1}{z^*\sqrt{n|h''(z^*)|}}e^{-nh(z^*)}\frac{1}{2\pi}\int_{-\delta\sqrt{n|h''(z^*)|}}^{\delta\sqrt{n|h''(z^*)|}}e^{-\frac{w^2}{2}}\\
&\frac{1}{1+\frac{\mathrm{i}w}{z^*\sqrt{n|h''(z^*)|}}}e^{\frac{h'''(z^*)}{6\sqrt{n}|h''(z^*)|^{3/2}}w^3\mathrm{i}-\frac{|h^{(4)}(\xi)|}{4!n|h''(z^*)|^2}w^4}dw
\end{align*}
By inequality
$$
\left|\frac{1}{1-x}-1-x\right|\leq \frac{|x|^2}{1-|x|}, ~~\left|e^{ix}-\sum_{l=0}^{m}\frac{(ix)^l}{l!}\right|\leq \min\left\{\frac{|x|^{m+1}}{(m+1)!},\frac{2|x|^m}{m!}\right\} 
$$
we know 
\begin{align*}
&\frac{1}{1+\frac{\mathrm{i}w}{z^*\sqrt{n|h''(z^*)|}}}e^{\frac{h'''(z^*)}{6\sqrt{n}|h''(z^*)|^{3/2}}w^3\mathrm{i}-\frac{|h^{(4)}(\xi)|}{4!n|h''(z^*)|^2}w^4}\\
=&1-\frac{\mathrm{i}w}{z^*\sqrt{n|h''(z^*)|}}+\frac{h'''(z^*)}{6\sqrt{n}|h''(z^*)|^{3/2}}w^3\mathrm{i}+O(\frac{1}{n})|w|^2+\cdots
\end{align*}
Therefore,
\begin{align*}
C=&\frac{1}{z^*\sqrt{2\pi n|h''(z^*)|}}e^{-nh(z^*)}\left[\frac{1}{\sqrt{2\pi}}\int_{-\delta\sqrt{n|h''(z^*)|}}^{\delta\sqrt{n|h''(z^*)|}}e^{-\frac{w^2}{2}}+O(\frac{1}{n})\right]\\
=&\frac{1}{z^*\sqrt{2\pi n|h''(z^*)|}}e^{-nh(z^*)}\left[1+O(\frac{1}{n})\right]
\end{align*}

Combining the estimations of $A$, $B$ and $C$ yields 
$$
\frac{1}{2 \pi \mathrm{i} } \int_{z^*- \mathrm{i} \infty}^{z^*+\mathrm{i} \infty} \frac{\left[1-(1-z)^\alpha\right]^k}{z^{n+1}} d z=\frac{1}{z^*\sqrt{2\pi n|h''(z^*)|}}e^{-nh(z^*)}\left[1+O(\frac{1}{n})\right].
$$
This completes the proof of Lemma \ref{Asymptotic_Steepest_descent_integral}.
 \end{proof}
 
 \begin{lemma}\label{MDP_rate}
Let $I(x)=h(z(x))$ and $x_k=\frac{k}{n}$, then 
\begin{align*}
(1-z(x_k))=&\alpha^{\frac{1}{1-\alpha}} x_k^{\frac{1}{1-\alpha}}\left[1+O(x_k^{\frac{\alpha}{1-\alpha}})\right]\\
I'(x)=&-\ln(1-(1-z(x)^{\alpha})),\\
I''(x)=&-\frac{\alpha(1-z(x))^{\alpha-1}}{1-(1-z(x))^{\alpha}}z'(x)\sim \frac{\alpha^{\frac{1}{1-\alpha}}}{1-\alpha}x^{\frac{2\alpha-1}{1-\alpha}}\\
h''(z(x))=&\frac{1-\alpha}{\alpha^{\frac{1}{1-\alpha}}}\frac{1}{x_k^{\frac{1}{1-\alpha}}}\left[1+O(x_k^{\frac{\alpha}{1-\alpha}})\right]
\end{align*}
\end{lemma}
\begin{proof}
Let $z(x_k)$ be the solution of the following critical equations
\begin{align*}
&0=h'(z)=\frac{1}{z}-x_k\frac{\alpha (1-z(x_k))^{\alpha-1}}{1-(1-z(x_k))^{\alpha}}\\
\iff& 0=1-(1-z(x_k))^{\alpha}-zx_k\alpha (1-z(x_k))^{\alpha-1}
\end{align*}
Then
\begin{align*}
(1-z(x_k))=&\left[\alpha x_k\frac{z(x_k)}{1-(1-z(x_k))^{\alpha}}\right]^{\frac{1}{1-\alpha}}\\
=&\alpha^{\frac{1}{1-\alpha}} x_k^{\frac{1}{1-\alpha}}\left[1+(1-z(x_k))^{\alpha}\frac{1-(1-z(x_k))^{1-\alpha}}{1-(1-z(x_k))^{\alpha}}\right]^{\frac{1}{1-\alpha}}\\
=&\alpha^{\frac{1}{1-\alpha}} x_k^{\frac{1}{1-\alpha}}\left[1+O(\alpha^{\frac{\alpha}{1-\alpha}} x_k^{\frac{\alpha}{1-\alpha}})\right]\\
z'(x_k)=&\frac{\alpha z(x_k)(1-z(x_k))^{\alpha}}{(1-\alpha x)(1-z(x_k))^{\alpha}-(1-\alpha)}
\end{align*}
and
\begin{align*}
I'(x)=&h'(z(x))z'(x)-\ln(1-(1-z(x))^{\alpha})=-\ln(1-(1-z(x))^{\alpha})\\
I''(x)=&-\frac{\alpha(1-z(x))^{\alpha-1}}{1-(1-z(x))^{\alpha}}z'(x)\sim \frac{\alpha^{\frac{1}{1-\alpha}}}{1-\alpha}x^{\frac{2\alpha-1}{1-\alpha}}
\end{align*}
Moreover,
\begin{align*}
 h^{\prime \prime}(z)=&-\frac{1}{z^2}-x_k\frac{-\left(1-(1-z)^\alpha\right)\alpha(\alpha-1)(1-z)^{\alpha-2}-\alpha^2(1-z)^{2 \alpha-2}}{\left(1-(1-z)^\alpha\right)^2}\\
 =&-\alpha(1-\alpha)x_k(1-z)^{\alpha-2}\left[\frac{1}{1-(1-z)^{\alpha}}+\frac{(1-z)^{2-\alpha}}{z^2}-\frac{(1-z)^{\alpha}}{(1-\alpha)(1-(1-z)^{\alpha})^2}\right]\\
 =&-\alpha(1-\alpha)x_k(1-z)^{\alpha-2}\left[1+\frac{(1-z)^{\alpha}}{1-(1-z)^{\alpha}}+\frac{(1-z)^{2-\alpha}}{z^2}-\frac{(1-z)^{\alpha}}{(1-\alpha)(1-(1-z)^{\alpha})^2}\right]\\
 =&-\alpha(1-\alpha)x_k(1-z)^{\alpha-2}\left[1+O((1-z)^{\alpha})\right]
 \end{align*}
Therefore
$$
 h^{\prime \prime}(z(x_k))=-\frac{1-\alpha}{\alpha^{\frac{1}{1-\alpha}}}\frac{1}{x_k^{\frac{1}{1-\alpha}}}\left[1+O(x_{k}^{\frac{\alpha}{1-\alpha}})\right].
 $$
\end{proof}

Making use of Lemma \ref{integral_representation}, Lemma \ref{coefficient_asymptotic}, Lemma \ref{Steepest_descent_deformation} and Lemma \ref{Asymptotic_Steepest_descent_integral}, we have the following local precise devation results.
\begin{theorem}\label{precise_local_DP}
Let $x_k=\frac{k}{n}$, $h(z)=\ln z-x_k\ln(1-(1-z)^{\alpha}), z\in(0,1)$ and $z(x_k)$ is the only root of $h'(z)=0$ for $z\in(0,1)$. 

\begin{itemize}
\item[(a)] Define $I(x_k)=h(z(x_k))$, then $I(x_k)$ is increasing in $x_k$. 
\item[(b)] \textbf{Precise local large deviations}: as $n\to\infty$, 
$$
\mathbb{P}(K_n=k)=\frac{\Gamma(\theta)}{\Gamma(\theta/\alpha)}\frac{n^{(\frac{1}{\alpha}-1)\theta-\frac{1}{2}}}{z(x_k)\sqrt{2\pi |h''(z(x_k))|}}x_k^{\frac{\theta}{\alpha}-1}e^{-nI(x_k)}\left[1+O(\frac{1}{x_k}\frac{1}{n})\right]
$$
\item[(c)] \textbf{Precise local moderate deviations}: Let $b_n>0$ and $\lim_{n\to\infty}b_n=\infty, \lim_{n\to\infty}\frac{b_n}{n}=0$. Then $x_k=\frac{k}{n}=\frac{y_k n^{\alpha}b_n^{1-\alpha}}{n}=y_k(\frac{b_n}{n})^{1-\alpha}$ where $y_k\in(0,\infty)$ and 
\begin{align*}
\mathbb{P}(K_n=k)=&\frac{\Gamma(\theta)}{\Gamma(\theta/\alpha)\sqrt{2\pi(1-\alpha)/\alpha^{1/(1-\alpha)}}}n^{(\frac{1}{\alpha}-1)\theta-\frac{1}{2}}\left(\frac{b_n}{n}\right)^{(1-\alpha)(\frac{\theta}{\alpha}-1)+\frac{1}{2}}\\
&y_k^{\frac{\theta}{\alpha}+\frac{1}{2(1-\alpha)}-1}e^{-nI(y_k(\frac{b_n}{n})^{1-\alpha})}\left[1+O(y^{\frac{\alpha}{1-\alpha}}(\frac{b_n}{n})^{\alpha}\vee\frac{1}{y_k}\frac{1}{b_n}(\frac{b_n}{n})^{\alpha})\right]
\end{align*}
\end{itemize}
\end{theorem}
\begin{rmk}
\item[(1)] The leading term in $nI(y_k(\frac{b_n}{n})^{1-\alpha})$ is 
$
b_n(1-\alpha)\alpha^{\frac{\alpha}{1-\alpha}}y_k^{\frac{1}{1-\alpha}},
$
which is the rate function of the moderate deviation for $K_n$ obtained in \cite{FFG}. The speed is $b_n$.
\item[(2)] The function $I(x_k)=h(z(x_k))=\sup_{z\in(0,1)}h(z)$. Consider a transformation $z=1-(1-e^{-\lambda})^{\frac{1}{\alpha}}$, then 
$$
I(x_k)=\sup_{\lambda>0}\{\lambda x_k+\ln[1-(1-e^{-\lambda})^{\frac{1}{\alpha}}]\},
$$
which is the rate function of the large deviation for $K_n$ obtained in \cite{FH}. The speed is $n$.
\end{rmk}

\section{Global Precise Deviations}\label{global}

To derive the global precise deviations, we need to accumulate the estimations of the precise local deviations. Thus, we should consider the following summations
$$
\sum_{k=[an]}^{[bn]}\psi(x_k)e^{-nf(x_k)}
$$
where both $\psi(x)$ and $f(x)$ are continuous bounded functions on $[a,b]$, and $f(x)$ is also differentiable and monotone on $[a,b]$. We need a lemma in \cite{GQZ} to handle the above summations. For the sake of convenience, we state the lemma as follows.
 
\begin{lemma} \label{D-Laplace-int-asy-thm-1}
Let $\psi(\beta)$ and $f(\beta)$ be  functions defined on interval  $[a,b]$ with continuous derivates up to order $2$ in $(a,b)$.  Let  $[\alpha_1,\alpha_2]\subset   [a,b]$ satisfy  the following condition: for some positive constants $c_0$ and $c_1$, and for any $\beta\in[\alpha_1,\alpha_2]$,
\begin{equation}\label{D-Laplace-int-asy-thm-1-eq-0}
c_0\leq  |f^{\prime}(\beta)|\leq c_1,\quad c_0\leq   \psi(\beta) \leq c_1.
\end{equation}
Suppose that $f$ is increasing  on   $[\alpha_1,\alpha_2]$,  then,  for $x \in [\alpha_1,\alpha_2]$, we have
\begin{equation*}\label{D-Laplace-int-asy-thm-1-eq-1}
\begin{aligned}
\mathcal L_n(\psi,f, [x,\alpha_2])
=& \left(1 +O\left(\frac{1}{n}\right)\right)\frac{\psi(x)e^{-\{nx\}f '(x)}}{1-e^{-f'(x)}} e^{-nf(x)}.
\end{aligned}
\end{equation*}
where $\{x\}=\lceil x\rceil -x$. 
 \end{lemma}
 
 By Lemma \ref{D-Laplace-int-asy-thm-1}, we have the following global precise large deviations for $K_n$.
 
 \begin{theorem}
 As $n\to\infty$, we have 
 $$
\mathbb{P}(K_n\geq xn)=\frac{\Gamma(\theta)}{\Gamma(\theta/\alpha)}\frac{n^{(\frac{1}{\alpha}-1)\theta-\frac{1}{2}}}{z(x)\sqrt{2\pi |h''(z(x))|}}x^{\frac{\theta}{\alpha}-1}\frac{e^{-\{nx\}I'(x)}}{1-e^{-I'(x)}}e^{-nI(x)}\left[1+O(\frac{1}{x}\frac{1}{n})\right]
$$
 \end{theorem}
 
 \begin{proof}
By Theorem \ref{precise_local_DP}, we know 
\begin{align*}
\mathbb{P}(K_n=k)=&\frac{\Gamma(\theta)}{\Gamma(\theta/\alpha)}\frac{n^{(\frac{1}{\alpha}-1)\theta-\frac{1}{2}}}{z(x_k)\sqrt{2\pi |h''(z(x_k))|}}x_k^{\frac{\theta}{\alpha}-1}e^{-nI(x_k)}\left[1+O(\frac{1}{x_k}\frac{1}{n})\right]\\
=&\frac{\Gamma(\theta)}{\Gamma(\theta/\alpha)}\frac{n^{(\frac{1}{\alpha}-1)\theta-\frac{1}{2}}}{\sqrt{2\pi}}\psi(x_k)e^{-nI(x_k)}\left[1+O(\frac{1}{x_k}\frac{1}{n})\right],
\end{align*}
where $\psi(x)=\frac{1}{z(x)\sqrt{x^{\frac{1}{(1-\alpha)}}|h''(z(x))|}}x^{\frac{1}{2(1-\alpha)}+\frac{\theta}{\alpha}-1}$.
Then
\begin{align*}
&\mathbb{P}(K_n\geq xn)=\sum_{k=\lceil x n\rceil}^n\mathbb{P}(K_n=k)\\
=&\sum_{k=\lceil xn\rceil}^n\frac{\Gamma(\theta)}{\Gamma(\theta/\alpha)}\frac{n^{(\frac{1}{\alpha}-1)\theta-\frac{1}{2}}}{\sqrt{2\pi}}\psi(x_k)e^{-nI(x_k)}\left[1+O(\frac{1}{x_k}\frac{1}{n})\right]\\
=&\frac{\Gamma(\theta)}{\Gamma(\theta/\alpha)}\frac{n^{(\frac{1}{\alpha}-1)\theta-\frac{1}{2}}}{\sqrt{2\pi}}\sum_{k=\lceil xn\rceil}^n\psi(x_k)e^{-nI(x_k)}\\
&+\frac{\Gamma(\theta)}{\Gamma(\theta/\alpha)}\frac{n^{(\frac{1}{\alpha}-1)\theta-\frac{1}{2}}}{\sqrt{2\pi}}O\left(\frac{1}{n}\sum_{k=\lceil xn\rceil}^n\frac{1}{x_k}\psi(x_k)e^{-nI(x_k)}\right)
\end{align*}
and we just need to focus on the asymptotic of both 
\begin{equation}
\sum_{k=\lceil xn\rceil}^n\psi(x_k)e^{-nI(x_k)}~\text{and}~\sum_{k=\lceil xn\rceil}^n\frac{1}{x_k}\psi(x_k)e^{-nI(x_k)}.\label{discrete_summation}
\end{equation}
Note that $\psi(y)$ is a continuous function on $[x,1]$, therefore it is bounded. Now it suffices to deal with the following form
$$
\sum_{k=\lceil xn\rceil}^n\psi(x_k)e^{-nI(x_k)}
$$
because both terms in equation (\ref{discrete_summation}) can be expressed as the above form.
 
 By Lemma \ref{D-Laplace-int-asy-thm-1}, we consider the following decomposition:
\begin{align*}
&\sum_{k=\lceil xn\rceil}^n\psi(x_k)e^{-nI(x_k)}= \left(1 +O\left(\frac{1}{n}\right)\right)\frac{\psi(x)e^{-\{nx\}I'(x)}}{1-e^{-I'(x)}} e^{-nI(x)}.
\end{align*}
and
\begin{align*}
&\sum_{k=\lceil xn\rceil}^n\frac{1}{x_k}\psi(x_k)e^{-nI(x_k)}= \left(1 +O\left(\frac{1}{n}\right)\right)\frac{1}{x}\psi(x)\frac{e^{-\{nx\}I'(x)}}{1-e^{-I'(x)}} e^{-nI(x)}.
\end{align*}
Therefore,
\begin{align*}
&\mathbb{P}(K_n\geq xn)\\
=&\frac{\Gamma(\theta)}{\Gamma(\theta/\alpha)}\frac{n^{(\frac{1}{\alpha}-1)\theta-\frac{1}{2}}}{\sqrt{2\pi}}\left(1 +O\left(\frac{1}{n}\right)\right)\frac{\psi(x)e^{-\{nx\}I'(x)}}{1-e^{-I'(x)}} e^{-nI(x)}\\
&+\frac{\Gamma(\theta)}{\Gamma(\theta/\alpha)}\frac{n^{(\frac{1}{\alpha}-1)\theta-\frac{1}{2}}}{\sqrt{2\pi}}O\left(\frac{1}{n}\left(1 +O\left(\frac{1}{n}\right)\right)\frac{1}{x}\psi(x)\frac{e^{-\{nx\}I'(x)}}{1-e^{-I'(x)}} e^{-nI(x)}\right)\\
=&\frac{\Gamma(\theta)}{\Gamma(\theta/\alpha)}\frac{n^{(\frac{1}{\alpha}-1)\theta-\frac{1}{2}}}{\sqrt{2\pi}}\frac{\psi(x)e^{-\{nx\}I'(x)}}{1-e^{-I'(x)}} e^{-nI(x)}\left[1 +O(\frac{1}{x}\frac{1}{n})\right]
\end{align*}
\end{proof}

 \begin{theorem}
 For $x=y\frac{n^{\alpha}b_n^{1-\alpha}}{n}$, where $\lim_{n\to\infty}b_n=\infty$ and $\lim_{n\to\infty}\frac{b_n}{n}=0$, we have 
 \begin{align*}
&\mathbb{P}(K_n\geq yn^{\alpha}b_n^{1-\alpha})\\
=&\frac{\Gamma(\theta)}{\Gamma(\theta/\alpha)}\frac{b_n^{(1-\alpha)\frac{\theta}{\alpha}-\frac{1}{2}}}{\sqrt{2\pi(1-\alpha)\alpha^{\frac{2\alpha-1}{1-\alpha}}}}y^{\frac{\theta}{\alpha}-\frac{1}{2(1-\alpha)}}e^{-nI(y(\frac{b_n}{n})^{1-\alpha})}\left[1+O\left[(y^{\frac{\alpha}{1-\alpha}}\vee \frac{1}{y}\frac{1}{b_n})(\frac{b_n}{n})^{\alpha}\right]\right]
\end{align*}
 \end{theorem}
\begin{proof}
If $x=y\frac{n^{\alpha}b_n^{1-\alpha}}{n}$, then by Lemma \ref{MDP_rate}, we have
\begin{align*}
&\mathbb{P}(K_n\geq yn^{\alpha}b_n^{1-\alpha})=\mathbb{P}(K_n\geq nx)\\
=&\frac{\Gamma(\theta)}{\Gamma(\theta/\alpha)}\frac{n^{(\frac{1}{\alpha}-1)\theta-\frac{1}{2}}}{z(x)\sqrt{2\pi |h''(z(x))|}}x^{\frac{\theta}{\alpha}-1}\frac{e^{-\{nx\}I'(x)}}{1-e^{-I'(x)}}e^{-nI(x)}\left[1+O(\frac{1}{x}\frac{1}{n})\right]\\
=&\frac{\Gamma(\theta)}{\Gamma(\theta/\alpha)}\frac{n^{(\frac{1}{\alpha}-1)\theta-\frac{1}{2}}}{\sqrt{2\pi \frac{1-\alpha}{\alpha^{\frac{1}{1-\alpha}}}\frac{1}{x^{\frac{1}{1-\alpha}}}\left[1+O(x_{k}^{\frac{\alpha}{1-\alpha}})\right]}}\frac{1}{z(x)}x^{\frac{\theta}{\alpha}-1}\\
&\frac{e^{\{nx\}\ln(1-(1-z(x))^{\alpha})}}{(1-z(x))^{\alpha}}e^{-nI(x)}\left[1+O(\frac{1}{x}\frac{1}{n})\right]\\
=&\frac{\Gamma(\theta)}{\Gamma(\theta/\alpha)}\frac{n^{(\frac{1}{\alpha}-1)\theta-\frac{1}{2}}}{\sqrt{2\pi \frac{1-\alpha}{\alpha^{\frac{1}{1-\alpha}}}}}\frac{x^{\frac{\theta}{\alpha}+\frac{1}{2(1-\alpha)}-1}}{\alpha^{\frac{\alpha}{1-\alpha}}x^{\frac{\alpha}{1-\alpha}}}\\
&\frac{[1+O((1-z(x))^{\alpha})]}{[1+O(x^{\frac{1}{1-\alpha}})][1+O(x^{\frac{\alpha}{1-\alpha}})]}e^{-nI(x)}\left[1+O(\frac{1}{x}\frac{1}{n})\right]\\
=&\frac{\Gamma(\theta)}{\Gamma(\theta/\alpha)}\frac{n^{(\frac{1}{\alpha}-1)\theta-\frac{1}{2}}}{\sqrt{2\pi(1-\alpha)\alpha^{\frac{2\alpha-1}{1-\alpha}}}}x^{\frac{\theta}{\alpha}-\frac{1}{2(1-\alpha)}}e^{-nI(x)}[1+O(x^{\frac{\alpha}{1-\alpha}})]\left[1+O(\frac{1}{x}\frac{1}{n})\right]\\
=&\frac{\Gamma(\theta)}{\Gamma(\theta/\alpha)}\frac{n^{(\frac{1}{\alpha}-1)\theta-\frac{1}{2}}\left(\frac{b_n}{n}\right)^{(1-\alpha)\frac{\theta}{\alpha}-\frac{1}{2}}}{\sqrt{2\pi(1-\alpha)\alpha^{\frac{2\alpha-1}{1-\alpha}}}}y^{\frac{\theta}{\alpha}-\frac{1}{2(1-\alpha)}}\\
&e^{-nI(y(\frac{b_n}{n})^{1-\alpha})}\left[1+O\left[(y^{\frac{\alpha}{1-\alpha}}\vee \frac{1}{y}\frac{1}{b_n})(\frac{b_n}{n})^{\alpha}\right]\right]\\
=&\frac{\Gamma(\theta)}{\Gamma(\theta/\alpha)}\frac{b_n^{(1-\alpha)\frac{\theta}{\alpha}-\frac{1}{2}}}{\sqrt{2\pi(1-\alpha)\alpha^{\frac{2\alpha-1}{1-\alpha}}}}y^{\frac{\theta}{\alpha}-\frac{1}{2(1-\alpha)}}e^{-nI(y(\frac{b_n}{n})^{1-\alpha})}\left[1+O\left[(y^{\frac{\alpha}{1-\alpha}}\vee \frac{1}{y}\frac{1}{b_n})(\frac{b_n}{n})^{\alpha}\right]\right]
\end{align*}
\end{proof}
 \begin{rmk}
\item[(1)] Here $b_n$ is the speed of moderate deviations. It only satisfies $\lim_{n\to\infty}b_n=\infty$ and $\lim_{n\to\infty}\frac{b_n}{n}=0$. When we expand the function $nI\big(y(\frac{b_n}{n})^{1-\alpha}\big)$, the leading term is $b_n(1-\alpha)\alpha^{\frac{\alpha}{1-\alpha}} y^{\frac{1}{1-\alpha}}$, which is the rate function of the moderate deviations for $K_n$; that is,
$$
\lim_{n\to\infty}\frac{1}{b_n}\ln \mathbb{P}\big(K_n\geq n^{\alpha}b_n^{1-\alpha}y\big)=-(1-\alpha)\alpha^{\frac{\alpha}{1-\alpha}} y^{\frac{1}{1-\alpha}}, \quad y>0.
$$
Note that no further restriction on $b_n$ is required, which differs from the setting in \cite{FFG}.

\item[(2)] Besides the leading term $b_n(1-\alpha)\alpha^{\frac{\alpha}{1-\alpha}} y^{\frac{1}{1-\alpha}}$, there are also non‑vanishing terms such as
$$
n x^{\frac{(l-1)\alpha+1}{1-\alpha}}=y^{\frac{(l-1)\alpha+1}{1-\alpha}}n\Big(\frac{b_n}{n}\Big)^{(l-1)\alpha+1},\qquad l\geq2.
$$
If $n\big(\frac{b_n}{n}\big)^{(l-1)\alpha+1}\to0$, then
$$
\lim_{n\to\infty} \frac{b_n}{n^{\frac{(l-1)\alpha}{(l-1)\alpha+1}}}=0.
$$
Thus the number of non‑vanishing terms depends on the particular scale of $b_n$.

\item[(3)] One can also observe a slight difference between the cases $\alpha\in(0,\frac{1}{2})$ and $\alpha\in(\frac{1}{2},1)$, because the factor $\alpha^{\frac{2\alpha-1}{1-\alpha}}$ in the constant coefficient
$$
\frac{\Gamma(\theta)}{\Gamma(\theta/\alpha)}\frac{1}{\sqrt{2\pi(1-\alpha)\alpha^{\frac{2\alpha-1}{1-\alpha}}}}
$$
behaves differently in the two regimes.
 \end{rmk}
 
 \section{Concluding Remarks}

Let us recall the fluctuation distribution $S_{\alpha,\theta}$.  Its density function is 
\begin{equation}
f_{S_{\alpha,\theta}}(s)=\frac{\Gamma(\theta+1)}{\alpha\Gamma(\frac{\theta}{\alpha}+1)}s^{\frac{\theta-1}{\alpha}-1}f_{\alpha}(s^{-1/\alpha})\mathbb{I}_{\{s>0\}} \label{fluctuation_distribution}
\end{equation}
where $f_{\alpha}(z)$ is the positive $\alpha$-stable density,
$$
f_{\alpha}(z)=\frac{1}{\pi}\sum_{j=1}^{\infty}\frac{(-1)^{j+1}}{j!}\sin(\pi\alpha j)\frac{\Gamma(\alpha j+1)}{z^{\alpha j+1}}\mathbb{I}_{\{z>0\}}.
$$
There are also many different equivalent forms (see \cite{Zol}), such as 
$$
f_{\alpha}(z)=\frac{1}{\pi}(\frac{\alpha}{1-\alpha})\left(\frac{1}{z}\right)^{\frac{1}{1-\alpha}}\int_0^{\pi}A(\varphi)\exp\left\{-\left(\frac{1}{z}\right)^{\frac{\alpha}{1-\alpha}}A(\varphi)\right\}d\varphi
$$
where $A(\varphi)=(\frac{\sin(\alpha\varphi)}{\sin\varphi})^{\frac{1}{1-\alpha}}(\frac{\sin((1-\alpha)\varphi)}{\sin(\alpha\varphi)})$. Then we shall derive the tail asymptotic of $S_{\alpha,\theta}$. First, we have 
$$
f_{S_{\alpha,\theta}}(s)=\frac{\Gamma(\theta+1)}{\pi(1-\alpha)\Gamma(\frac{\theta}{\alpha}+1)}s^{\frac{\theta-1}{\alpha}+\frac{1}{\alpha(1-\alpha)}-1}\int_0^{\pi}A(\varphi)\exp\left\{-s^{\frac{1}{1-\alpha}}A(\varphi)\right\}d\varphi \mathbb{I}_{\{s>0\}}
$$
Then the tail is 
\begin{align*}
\int_{x}^{\infty}f_{S_{\alpha,\theta}}(s)ds=\frac{\Gamma(\theta+1)}{\pi(1-\alpha)\Gamma(\frac{\theta}{\alpha}+1)}\int_0^{\pi}
\left[\int_x^{\infty}s^{\frac{\theta-1}{\alpha}+\frac{1}{\alpha(1-\alpha)}-1}A(\varphi)\exp\left\{-s^{\frac{1}{1-\alpha}}A(\varphi)\right\}ds\right] d\varphi
\end{align*}

\begin{lemma}\label{tail_asymptotic_fluctuation}
As $x\to\infty$, 
$$
\mathbb{P}(S_{\alpha,\theta}\geq x)\sim \frac{\Gamma(\theta)}{\Gamma(\theta/\alpha)\pi}\alpha x^{\frac{\theta}{\alpha}}e^{-(1-\alpha)\alpha^{\frac{\alpha}{1-\alpha}}x^{\frac{1}{1-\alpha}}}
$$
\end{lemma}
\begin{proof}
Note that
\begin{align*}
\mathbb{P}(S_{\alpha,\theta}\geq x)=&\frac{\Gamma(\theta+1)}{\pi(1-\alpha)\Gamma(\frac{\theta}{\alpha}+1)}\int_0^{\pi}
\left[\int_x^{\infty}s^{\frac{\theta-1}{\alpha}+\frac{1}{\alpha(1-\alpha)}-1}A(\varphi)\exp\left\{-s^{\frac{1}{1-\alpha}}A(\varphi)\right\}ds\right] d\varphi
\end{align*}
By substitution $w=s^{\frac{1}{1-\alpha}}A(\varphi)-x^{\frac{1}{1-\alpha}}A(\varphi)$, we have
\begin{align*}
&\int_x^{\infty}s^{\frac{\theta-1}{\alpha}+\frac{1}{\alpha(1-\alpha)}-1}A(\varphi)\exp\left\{-s^{\frac{1}{1-\alpha}}A(\varphi)\right\}ds\\
=&(1-\alpha)x^{\frac{\theta-1}{\alpha}+\frac{1}{\alpha}}e^{-x^{\frac{1}{1-\alpha}}A(\varphi)}\int_0^{\infty}\left(1+\frac{w}{x^{\frac{1}{1-\alpha}}A(\varphi)}\right)^{\frac{(1-\alpha)(\theta-1)}{\alpha}+\frac{1-\alpha}{\alpha}}e^{-w}dw
\end{align*}
Denote $B(\varphi,x)=\int_0^{\infty}\left(1+\frac{w}{x^{\frac{1}{1-\alpha}}A(\varphi)}\right)^{\frac{(1-\alpha)(\theta-1)}{\alpha}+\frac{1-\alpha}{\alpha}}e^{-w}dw$, then
\begin{align*}
\mathbb{P}(S_{\alpha,\theta}\geq x)=&\frac{\Gamma(\theta+1)}{\pi(1-\alpha)\Gamma(\frac{\theta}{\alpha}+1)}(1-\alpha)x^{\frac{\theta-1}{\alpha}+\frac{1}{\alpha}}\int_0^{\pi}
B(\varphi,x)e^{-x^{\frac{1}{1-\alpha}}A(\varphi)}d\varphi.
\end{align*}
Since $\lim_{x\to\infty}B(\varphi,x)=1$ and $\int_0^{\pi}
B(\varphi,x)e^{-x^{\frac{1}{1-\alpha}}A(\varphi)}d\varphi\sim e^{-x^{\frac{1}{1-\alpha}}A(0)}=e^{-(1-\alpha)\alpha^{\frac{\alpha}{1-\alpha}}x^{\frac{1}{1-\alpha}}}$, then 
\begin{align*}
\mathbb{P}(S_{\alpha,\theta}\geq x)\sim& \frac{\Gamma(\theta+1)}{\pi(1-\alpha)\Gamma(\frac{\theta}{\alpha}+1)}(1-\alpha)x^{\frac{\theta-1}{\alpha}+\frac{1}{\alpha}}e^{-(1-\alpha)\alpha^{\frac{\alpha}{1-\alpha}}x^{\frac{1}{1-\alpha}}}\\
=&\frac{\Gamma(\theta)}{\Gamma(\theta/\alpha)\pi}\alpha x^{\frac{\theta}{\alpha}}e^{-(1-\alpha)\alpha^{\frac{\alpha}{1-\alpha}}x^{\frac{1}{1-\alpha}}}
\end{align*}
\end{proof}

Thus, we observe a similar decay rate function for both $S_{\alpha,\theta}$ and $\frac{K_n}{n^{\alpha}}$. This also suggests the possibility of further investigation into Cramér-type moderate deviations.

\end{document}